\documentclass{article}

\pdfoutput=1

\usepackage{epsfig,amsmath,amssymb,graphics,verbatim,amsfonts,subfigure,psfrag,amsthm}
\usepackage[letterpaper,margin=1in]{geometry}
\newtheorem{thm}{Theorem}[section]

\newtheorem{prop}[thm]{Proposition}

\newtheorem{cjt}[thm]{Conjecture}

\def\R{\mathbb{R}}

\def\N{\mathbb{N}}

\newcommand{\be}{\begin{equation}}
\newcommand{\ee}{\end{equation}}
\newcommand{\bea}{\begin{eqnarray}}
\newcommand{\eea}{\end{eqnarray}}
\newcommand{\beann}{\begin{eqnarray*}}
\newcommand{\eeann}{\end{eqnarray*}}
\newcommand{\benn}{\begin{equation*}}
\newcommand{\eenn}{\end{equation*}}

\def\ra{\rightarrow}

\begin{document}
 
\date{}
\title{Homoclinic Orbits of the FitzHugh-Nagumo Equation: Bifurcations in the
Full System}
\author{John Guckenheimer\thanks{Mathematics Department, Cornell University}
\and Christian Kuehn\thanks{Center for Applied Mathematics, Cornell University}}

\maketitle

\begin{abstract}
This paper investigates travelling wave solutions of the FitzHugh-Nagumo
equation from the viewpoint of fast-slow dynamical systems. These solutions are
homoclinic orbits of a three dimensional vector field depending upon system
parameters of the FitzHugh-Nagumo model and the wave speed. Champneys et al.
[A.R. Champneys, V.~Kirk, E.~Knobloch, B.E. Oldeman, and J.~Sneyd, When
Shilnikov meets Hopf in excitable systems, \textit{SIAM Journal of Applied
Dynamical Systems}, 6(4), 2007] observed sharp turns in the curves of homoclinic
bifurcations in a two dimensional  parameter space. This paper demonstrates
numerically that these turns are located close to the intersection of two curves
in the parameter space that locate non-transversal intersections of invariant
manifolds of the three dimensional vector field. The relevant invariant
manifolds in phase space are visualized. A geometrical model inspired by the
numerical studies displays the sharp turns of the homoclinic bifurcations curves
and yields quantitative predictions about multi-pulse and homoclinic orbits and
periodic orbits that have not been resolved in the FitzHugh-Nagumo model.
Further observations address the existence of canard explosions and mixed-mode
oscillations.
\end{abstract}
 
\section{Introduction}

This paper investigates the three dimensional FitzHugh-Nagumo vector field
defined by:
\bea
\label{eq:fhn}
\epsilon\dot{x}_1&=&x_2\nonumber\\
\epsilon\dot{x}_2&=&\frac{1}{\Delta}
\left(sx_2-x_1(x_1-1)(\alpha-x_1)+y-p\right)=:\frac1\Delta\left(sx_2-f(x_1)+y-p
\right)\\
\dot{y}&=&\frac1s \left(x_1-y\right)\nonumber
\eea   
where $p$, $s$, $\Delta$, $\alpha$ and $\epsilon$ are parameters. Our analysis
views equations \eqref{eq:fhn} as a fast-slow system with two fast variables and
one slow variable. The dynamics of system \eqref{eq:fhn} were studied
extensively by Champneys et al.~\cite{Sneydetal} with an emphasis on homoclinic
orbits that represent travelling wave profiles of a partial differential
equation~\cite{AronsonWeinberger}. Champneys et al.~\cite{Sneydetal} used
numerical continuation implemented in AUTO \cite{Doedel_AUTO2007} to analyze the
bifurcations of \eqref{eq:fhn} for $\epsilon=0.01$ with varying $p$ and $s$. As
in their studies, we choose $\Delta = 5$, $\alpha = 1/10$ for the numerical
investigations in this paper. The main structure of the bifurcation diagram is
shown in Figure \ref{fig:cu_improved}.\\

\begin{figure}[htbp]
\centering
\includegraphics[width=0.8\textwidth]{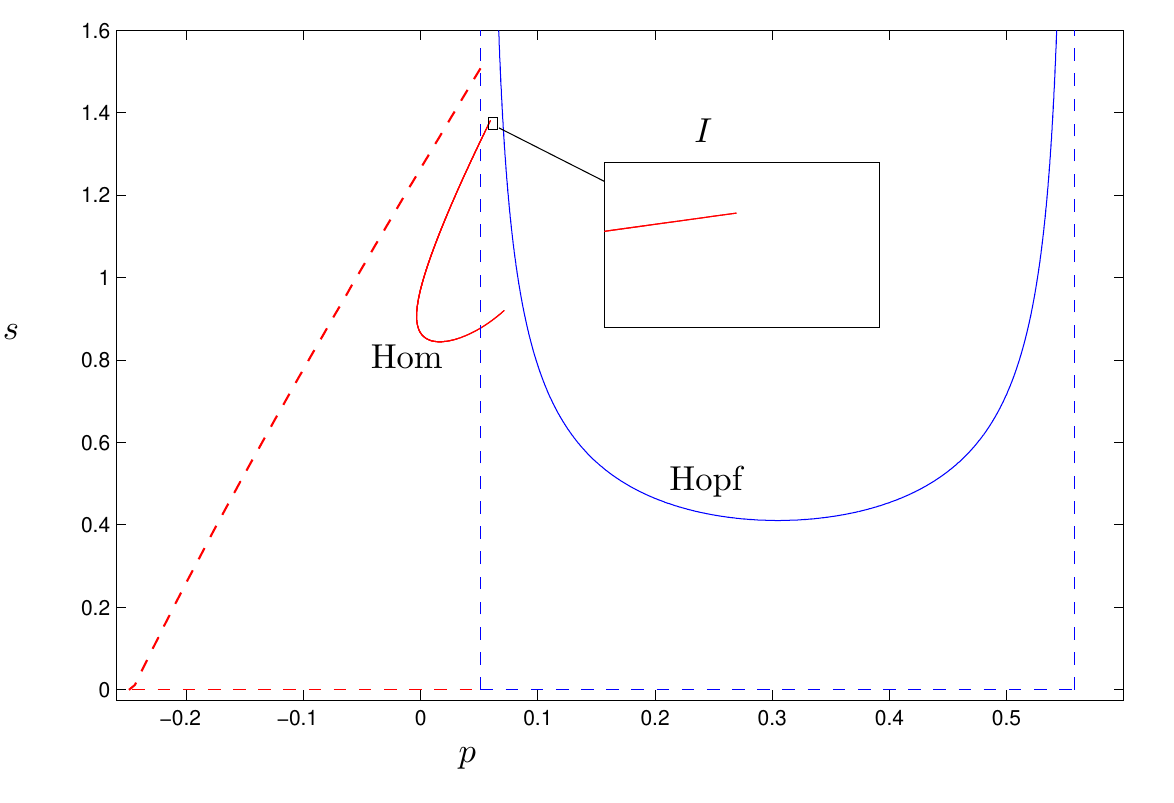}  
\caption{Bifurcation diagram for the FitzHugh-Nagumo equation \eqref{eq:fhn}.
Shilnikov homoclinic bifurcations (solid red) and Hopf bifurcations (solid blue)
are shown for $\epsilon=0.01$. The dashed curves show the singular limit
($\epsilon=0$) bifurcation curves for the homoclinic and Hopf bifurcations; see
\cite{GuckCK1} and Section \ref{sec:basic} for details on the singular limit
part of the diagram.}
 \label{fig:cu_improved}
\end{figure}

Figure \ref{fig:cu_improved} shows a curve of Hopf bifurcations which is
U-shaped and a curve of Shilnikov homoclinic bifurcations which is C-shaped.
Champneys et al. \cite{Sneydetal} observed that the C-curve is a closed curve
which folds back onto itself before it reaches the U-curve, and they discussed
bifurcations that can ``terminate'' a curve of homoclinic bifurcations. Their
analysis does not take into account the multiple-time scales of the vector field
\eqref{eq:fhn}. This paper demonstrates that fast-slow analysis of the
homoclinic curve yields deeper insight into the events that occur at the sharp
turn of the homoclinic curve. We shall focus on the turning point at the top end
of the C-curve and denote this region by $I$.\\

We regard $\epsilon$ in the FitzHugh-Nagumo equation \eqref{eq:fhn} as a small
parameter. In \cite{GuckCK1}, we derived a singular bifurcation diagram which
represents several important bifurcation curves in $(p,s)$-parameter space in
the singular limit $\epsilon=0$. The singular limits of the Hopf and homoclinic
curves are shown in Figure \ref{fig:cu_improved} as dotted lines.\footnote{In
Section \ref{sec:basic} we recall the precise meaning of the singular limit
bifurcation from \cite{GuckCK1} and how they these bifurcations arise when
$\epsilon=0$.} In the singular limit, there is no gap between the Hopf and
homoclinic curves. We demonstrate below in Proposition 2.1 that a gap must
appear for $\epsilon>0$. The main point of this paper is that the termination
point of the C-curve at the end of the gap is due to a fast-slow ``bifurcation''
where the two dimensional stable manifold of an equilibrium is tangent to the
two dimensional unstable manifold of a one dimensional slow 
manifold.\footnote{An analogous tangency plays an important role in the
formation of mixed mode oscillations associated with singular Hopf bifurcations
in fast-slow systems with one fast and two slow
variables~\cite{GuckenheimerSH}.} Since the analysis of \cite{Sneydetal} does
not explicitly consider slow manifolds of the system, this tangency does not
appear in their list of possibilities for the termination of a C-curve. Note
that the slow manifolds of the system are unique only up to ``exponentially
small'' quantities of the form $\exp(-c/\epsilon), c>0$, so our analysis only
identifies the termination point up to exponentially small values of the
parameters.\\

Fast-slow dynamical systems can be written in the form
\bea
\label{eq:basic1}
\epsilon \dot{x}&=&\epsilon\frac{dx}{d\tau}=f(x,y,\epsilon)\\
\dot{y}&=&\frac{dy}{d\tau}=g(x,y,\epsilon)\nonumber
\eea 
where $(x,y)\in\R^m\times \R^n$ and $\epsilon$ is a small parameter
$0<\epsilon\ll 1$. The functions $f:\R^m\times\R^n\times \R\ra \R^m$ and
$g:\R^m\times\R^n\times \R\ra \R^n$ are analytic in the systems studied in this
paper. The variables $x$ are fast and the variables $y$ are slow. We can
change \eqref{eq:basic1} from the slow time scale $\tau$ to the fast time
scale $t=\tau/\epsilon$, yielding
\bea
\label{eq:basic2}
x'&=&\frac{dx}{dt}=f(x,y,\epsilon)\\
y'&=&\frac{dy}{dt}=\epsilon g(x,y,\epsilon)\nonumber
\eea 
In the singular limit $\epsilon\ra 0$ the system \eqref{eq:basic1} becomes a
differential-algebraic equation. The algebraic constraint defines the critical
manifold:
\benn
C_0=\{(x,y)\in\R^m\times \R^n:f(x,y,0)=0\}
\eenn
For a point $p\in C_0$ we say that $C_0$ is normally hyperbolic at $p$ if the
all the  eigenvalues of the $m\times m$ matrix $D_xf(p)$ have non-zero real
parts. A normally hyperbolic subset of $C_0$ is an actual manifold and we can
locally parametrize it by a function $h(y)=x$. This yields the slow subsystem
(or reduced flow) $\dot{y}=g(h(y),y)$ defined on $C_0$. Taking the singular
limit $\epsilon\ra 0$ in \eqref{eq:basic2} gives the fast subsystem (or layer
equations) $x'=f(x,y)$ with the slow variables $y$ acting as parameters.
Fenichel's Theorem \cite{Fenichel4} states that normally hyperbolic critical
manifolds perturb to invariant slow manifolds $C_\epsilon$. A slow manifold
$C_\epsilon$ is $O(\epsilon)$ distance away from $C_0$. The flow on the
(locally) invariant manifold $C_\epsilon$ converges to the slow subsystem on the
critical manifold as $\epsilon\ra 0$. Slow manifolds are usually not unique for
a fixed value of $\epsilon=\epsilon_0$ but lie at a distance
$O(e^{-K/\epsilon_0})$ away from each other for some $K>0$; nevertheless we
shall refer to ``the slow manifold'' for a fast-slow system with the possibility
of an exponentially small error being understood.\\

Section \ref{sec:basic} discusses the fast-slow decomposition of the homoclinic
orbits of the FitzHugh-Nagumo equation in the region $I$. This decomposition has
been used to prove the existence of homoclinic orbits in the system for
$\epsilon$ sufficiently
small~\cite{Carpenter,Hastings1,JonesFHN,JonesKopellLanger,KSS1997}, but
previous work only applies to a situation where the equilibrium point for the
homoclinic orbit is not close to a fold point. At a fold point the critical
manifold of a fast-slow system is locally quadratic and not normally hyperbolic.
This new aspect of the decomposition is key to understanding the sharp turn of
the homoclinic curve. Section \ref{sec:turning} presents a numerical study that
highlights the geometric mechanism for the turning of the C-curve. We visualize
relevant aspects of the phase portraits near the turns of the C-curve. In
Section \ref{sec:contract} we show that exponential contraction of the Shilnikov
return map in the FitzHugh-Nagumo equation explains why n-homoclinic and
n-periodic orbits are expected to be found at parameter values very close to a
primary 1-homoclinic orbit. Section \ref{sec:further} presents two further
observations. We identify where a canard explosion \cite{KruSzm2} occurs and we
note the existence of two different types of mixed-mode oscillations in the
system.

\section{Fast-Slow Decomposition of Homoclinic Orbits}
\label{sec:basic}

We introduce notation used in our earlier work~\cite{GuckCK1}. The critical
manifold of \eqref{eq:fhn} is given by:
\benn
C_0=\{(x_1,x_2,y)\in\R^3:x_2=0\text{ and }y=f(x_1)+p\}
\eenn
It is normally hyperbolic away from the two fold points $x_{1,\pm}$ with
$x_{1,-}<x_{1,+}$ which are found by solving $f'(x_1)=0$ as the local minimum
and maximum of the cubic $f$. Hence $C_0$ splits into three parts:
\benn
C_l=\{x_1<x_{1,-}\}\cap C_0,\quad C_m=\{x_{1,-}\leq x_1 \leq x_{1,+}\}\cap
C_0,\quad C_r=\{x_{1,+}\}\cap C_0
\eenn
We are mostly interested in the two branches $C_l$ and $C_r$ which are of
saddle-type, i.e. points in $C_l$ and $C_r$ are saddle equilibria of the fast
subsystem. The middle branch $C_m-\{x_{1,\pm}\}$ consists of unstable foci for
the fast subsystem. The slow manifolds provided by Fenichel's Theorem will be
denoted by $C_{l,\epsilon}$ and $C_{r,\epsilon}$. The notation for the
two-dimensional stable and unstable manifolds of $C_{l,\epsilon}$ is
$W^s(C_{l,\epsilon})$ and $W^u(C_{r,\epsilon})$ with similar notation for
$C_{r,\epsilon}$; the notation for the associated linear eigenspaces is e.g.
$E^s(C_{l,\epsilon})$. The full system \eqref{eq:fhn} has a unique equilibrium
point which we denote by $q$. For $(p,s)\in I$ and $\epsilon=0.01$ the
dimensions of the stable and unstable manifolds are $\dim(W^u(q))=1$ and
$\dim(W^s(q))=2$ with a complex conjugate pair of eigenvalues for the
linearization at $q$. The equilibrium $q$ is completely unstable inside the
U-curve and the Hopf bifurcations we are interested in near $I$ are all
subcritical \cite{GuckCK1,Sneydetal}.\\

As $\epsilon\ra 0$ the Hopf bifurcation curve converges to a region in $(p,s)$
parameter space bounded by two vertical lines $p=p_\pm$ and the segment
$\{s=0,p_-\leq p\leq p_+\}$; see Figure \ref{fig:cu_improved}. The parameter
values $p_\pm$ are precisely the values when the equilibrium point $q$ coincides
with the fold points $x_{1,\pm}$ \cite{GuckCK1}. This analysis gives one part of
the singular limit bifurcation diagram showing what happens to the Hopf
bifurcation curves for $\epsilon=0$.\\

\begin{figure}[htbp]
\centering
\includegraphics[width=0.45\textwidth]{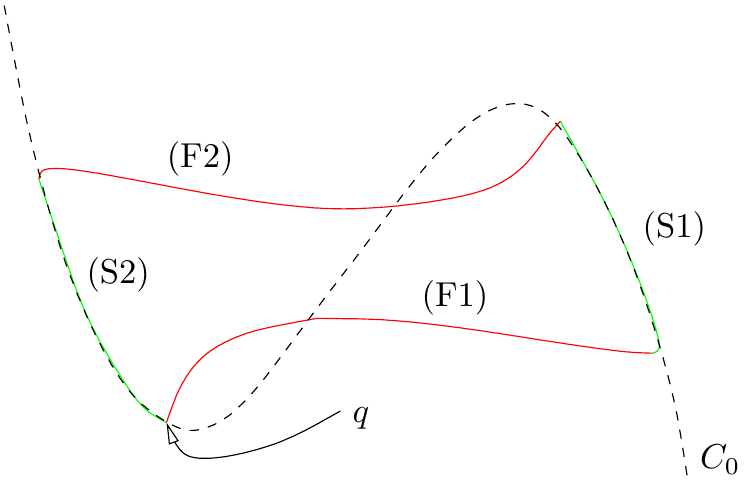}  
\caption{\label{fig:decomp}Sketch of a homoclinic orbit to the unique
equilibrium $q$. Fast (red) and slow (green) segments decompose the orbit into
segments.}
\end{figure}

When $\epsilon$ is small, the homoclinic orbit in $W^u(q) \cap W^s(q)$ can be
partitioned into fast and slow segments. The singular limit of this fast-slow
decomposition has four segments: a fast subsystem heteroclinic connection from
$q$ to $C_r$, a slow segment on $C_r$, a heteroclinic connection from $C_r$ to
$C_l$ and a slow segment back to $q$ on $C_l$; see Figure \ref{fig:decomp}.
Existence proofs for the homoclinic
orbits~\cite{JonesKopellLanger,Hastings1,Carpenter} are based upon analysis of
the transitions between these segments. Trajectories that remain close to a
normally hyperbolic slow manifold must be ``exponentially close'' to the
manifold except for short segments where the trajectory approaches the slow
manifold along its stable manifold and departs along its unstable manifold.
Existence of the homoclinic orbit depends upon how the four segments of its
fast-slow decomposition fit together:

\begin{enumerate}
\item[(F1)] The one dimensional $W^{u}(q)$ approaches $C_r$ along its two
dimensional stable manifold  $W^s(C_{r,\epsilon})$. Intersection of these
manifolds cannot be transverse and occurs only for parameter values that lie
along a curve in the $(p,s)$ parameter plane. 
\item[(S1)] The Exchange Lemma~\cite{JonesKopell} was developed to analyze the
flow map for trajectories that approach $C_{r,\epsilon}$ along its stable
manifold and depart $C_{r,\epsilon}$ along its unstable manifold. 
\item[(F2)] The fast jump from a neighborhood of $C_{r,\epsilon}$ to a
neighborhood of $C_{l,\epsilon}$ occurs along a transversal intersection of the
two dimensional $W^s(C_{l,\epsilon})$ and two dimensional $W^u(C_{r,\epsilon})$. 
\item[(S2)] The connection from $C_{l,\epsilon}$ to $q$ lies close to an
intersection of the two dimensional $W^u(C_{l,\epsilon})$ and the two
dimensional $W^s(q)$. Previous analysis has dealt with parameter regions where
the connection (S2) exists and is transversal, but it \emph{cannot} persist up
to the Hopf curve in the $(p,s)$-plane. 
\end{enumerate}
$\quad$\\

\begin{prop}\label{prop:Hopf} There exists a region in $(p,s)$-parameter space
near the Hopf U-curve where no trajectories close to $C_{l,\epsilon}$ lie in
$W^s(q)$. 
\end{prop}  

\begin{proof} (Sketch) The Lyapunov coefficients of the Hopf bifurcations near
$I$ are positive~\cite{GuckCK1}, so the periodic orbits emanating from these
bifurcations occur in the parameter region to the left of the Hopf curve. The
periodic orbits are completely unstable. By calculating the eigenvalues of the
linearization at the equilibrium we find that there is no Fold-Hopf bifurcation
on the Hopf curve near $I$. Hence center manifold reduction implies that there
will be a region of parameters near the Hopf curve where $W^{s}(q)$ is a
topological disk whose boundary is the periodic orbit. Close enough to the Hopf
curve, $W^{s}(q)$ and the periodic orbit lie at a finite distance from
$C_{l,\epsilon}$ and there is no connection from $C_{l,\epsilon}$ to $q$. 
\end{proof}
$\quad$\\

This proposition implies that the parameter region in which there is a
connection from $C_{l,\epsilon}$ to $q$ is bounded away from the Hopf curve. The
next section shows that the boundary of this parameter region is very close to a
curve along which there are tangential intersections of $W^u(C_{l,\epsilon})$
and $W^s(q)$.\\

\textit{Remark}: As $\epsilon\ra 0$, the C-curve converges to two lines (dashed
red in Figure \ref{fig:cu_improved}) defined by homoclinic and heteroclinic
orbits of the fast subsystem \cite{GuckCK1}. The horizontal segment of the
C-curve to homoclinic orbits of the equilibrium point, and the sloped segment
to heteroclinic orbits from the equilibrium point to the right branch of
the critical manifold. Note that the C-curve terminates on the Hopf curve
in the singular limit. The singular limit analysis does not explain the sharp
turning of the C-curve for $\epsilon>0$ which is the focus of the next section. 

\section{Interaction of Invariant Manifolds}
\label{sec:turning}

The slow manifold $C_{l,\epsilon}$ is normally hyperbolic away from the fold
point $x_{1,-}$, with one attracting direction and one repelling direction. We
recently introduced a method \cite{GuckCK2} for computing slow manifolds of
saddle type. This algorithm is used here to help determine whether there are
connecting orbits from a neighborhood of $C_{l,\epsilon}$ to the equilibrium
point $q$. Our numerical strategy for finding connecting orbits has three
steps: 
\begin{enumerate}
\item
Choose the cross section 
$$\Sigma_{0.09}=\{(x_1,x_2,y)\in\R^3:y=0.09\}$$ 
transverse to $C_{l,\epsilon}$, 
\item
Compute intersections of trajectories in $W^s(q)$ with $\Sigma_{0.09}$. These
points are found either by backward integration from initial conditions that lie
in a small disk $D$ containing $q$ in $W^s(q)$ or by solving a boundary value
problem for trajectories that have one end in $\Sigma_{0.09}$ and one end on the
boundary of $D$.
\item
Compute the intersection $p_l \in C_{l,\epsilon}\cap\Sigma_{0.09}$ with the
algorithm described in  Guckenheimer and Kuehn \cite{GuckCK2} and determine the
directions of the positive and negative eigenvectors of the Jacobian of the fast
subsystem at $p_l$.
\end{enumerate}
$\quad$\\

Figure \ref{fig:move} shows the result of these computations for $\epsilon =
0.01$, $s=1.37$ and three values of $p$.\footnote{The second step above was
carried out with two different initial value solvers, \textit{ode15s} in Matlab
\cite{ShampineReichelt} and dop853~\cite{HairerWannerII}, and with AUTO
\cite{Doedel_AUTO2007} used as a boundary value solver. All three methods
produced similar results.}  The intersections of $W^s(q)$ with $\Sigma_{0.09}$
lies close to $W^s(C_{l,\epsilon})$.\\

\begin{table}[htbp]
\begin{center}
\begin{tabular}{|c|c|}
\hline
$\epsilon$ & D=d(tangency,Hopf)\\
\hline  
$10^{-2}$ & $\approx 1.07\epsilon$ \\
$10^{-3}$ & $\approx 1.00\epsilon$ \\
$10^{-4}$ & $\approx 0.98\epsilon$ \\
\hline
\end{tabular}
\caption{Euclidean distance in (p,s)-parameter space between the Hopf curve and
the location of the tangency point between $W^s(q)$ and
$W^u(C_{l,\epsilon})$\label{tab:dist}.}
\end{center}
\end{table}

Backward trajectories flowing along $C_{l,\epsilon}$ converge to its stable
manifold at a fast exponential rate. This fact also explains the observation
that $W^s(q) \cap \Sigma_{0.09}$ makes a sharp turn. In Figure \ref{fig:move}(a), it
is apparent that the turn lies to the left of $W^u(C_{l,\epsilon}) \cap
\Sigma_{0.09}$ and that $W^s(q) \cap W^u(C_{l,\epsilon})$ is non-empty. In
Figure \ref{fig:move}(c), the turn lies to the right of $W^u(C_{l,\epsilon}) \cap
\Sigma_{0.09}$. We have also computed the distance from the Hopf curve of the
parameters at which $W^s(q)$ and $W^u(C_{l,\epsilon})$ appear to have a
tangential intersection for several different values of $\epsilon$; see Table
\ref{tab:dist} from which we observe that the distance is $O(\epsilon)$.\\

\begin{figure}[htbp]
\centering
\includegraphics[width=0.9\textwidth]{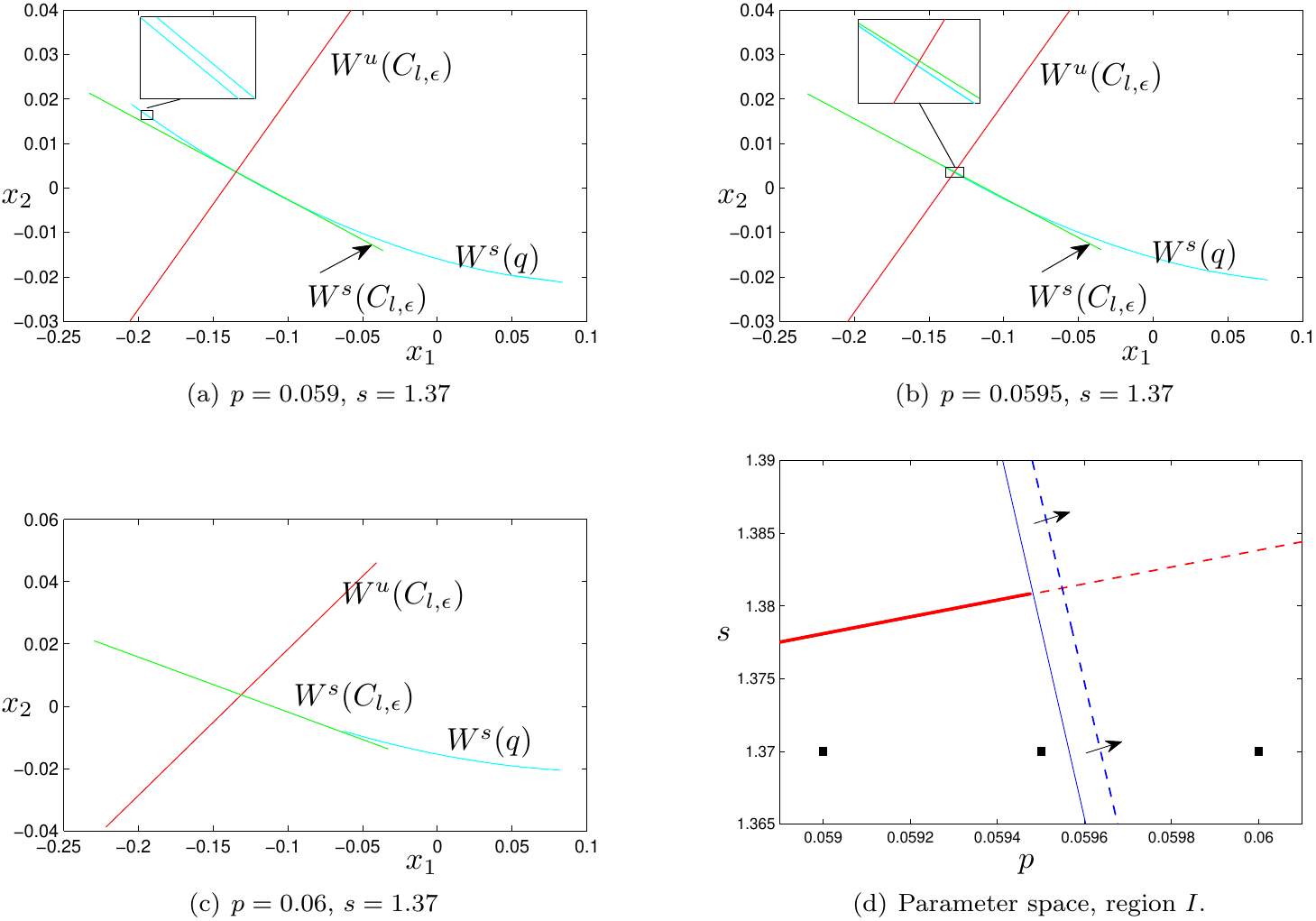}
\caption{\label{fig:move}Figures (a)-(c) show the movement of the stable
manifold $W^s(q)$ (cyan) with respect to $E^u(C_{l,\epsilon})$ (red) and
$E^s(C_{l,\epsilon})$ (green) in phase space on the section $y=0.09$ for
$\epsilon=0.01$. The parameter space diagram (d) shows the homoclinic C-curve
(solid red), an extension of the C-curve of parameters where $W^{u}(q) \cap
W^s(C_{r,\epsilon})$ is nonempty, a curve that marks the tangency of $W^s(q)$ to
$E^u(C_{l,\epsilon})$ (blue) and a curve that marks a distance between
$C_{l,\epsilon}$ and $W^s(q)$ (dashed blue) of $0.01$ where the arrows indicate
the direction in which the distance is bigger than $0.01$. The solid black
squares in (d) show the parameter values for (a)-(c).}
\end{figure}

In Figure \ref{fig:move}(d) the C-curve of homoclinic bifurcations (solid red) has
been computed using continuation in AUTO \cite{Doedel_AUTO2007} as carried out
by Champneys et al. \cite{Sneydetal}. Despite the fact that no homoclinic orbit
exists in part of the region $I$ it is possible to check whether the unstable
manifold $W^u(q)$ reaches a small neighborhood of $W^s(C_{r,\epsilon})$. This
idea has been used in a splitting algorithm \cite{GuckCK1} to calculate where
homoclinic orbits would occur if $W^s(q)$ would not move away from
$C_{l,\epsilon}$ as shown in Figures \ref{fig:move}(a)-\ref{fig:move}(d). This yields the
dashed red curve in Figure \ref{fig:move}(d). On this curve we verified that
$W^s(C_{l,\epsilon})$ and $W^u(C_{r,\epsilon})$ still intersect transversally by
computing those manifolds; see \cite{GuckCK1,GuckCK2} for details.\\

The blue curves in Figure \ref{fig:move}(d) have been obtained by measuring the
Euclidean distances between $W^s(q)$ and $C_{l,\epsilon}$ in the section
$\Sigma_{0.09}$. Along the dashed blue curve the distance between
$C_{l,\epsilon}$ and $W^s(q)$ is $0.01$.  The arrows indicate the direction in
which this distance increases. The solid blue curve marks a tangency of $W^s(q)$
with $E^u(C_{l,\epsilon})$. These calculations demonstrate that the sharp turn
in the C-curve of homoclinic bifurcations occurs very close to the curve where
there is a tangential intersection of  $W^s(q)$ and $W^u(C_{l,\epsilon})$.
Therefore, we state the following conjecture.\\

\begin{cjt}\label{prop:main} The C-curve of homoclinic bifurcations of the
FitzHugh-Nagumo system turns exponentially close to the boundary of the region
where $W^u(C_{l,\epsilon}) \cap W^s(q)$ is nonempty. 
\end{cjt} 
$\quad$\\

Note that trajectory segments of types (F1), (S1) and (F2) are still present
along the dashed red curve in Figure \ref{fig:move}(d). Only the last slow connection
(S2) no longer exists. Existence proofs for homoclinic orbits that use
Fenichel's Theorem for $C_{l}$ to conclude that trajectories entering a small
neighborhood of $C_{l,\epsilon}$ must intersect $W^s(q)$ break down in this
region. The equilibrium $q$ has already moved past the fold point $x_{1,-}$ in
$I$ as seen from the singular bifurcation diagram in Figure
\ref{fig:cu_improved} where the blue dashed vertical lines mark the parameter
values where $q$ passes through $x_{1,\pm}$. Therefore Fenichel's Theorem does
not provide the required perturbation of $C_{l,\epsilon}$. Previous proofs
\cite{JonesKopellLanger,Hastings1,Carpenter} assumed $p=0$ and the connecting
orbits of type (S2) do exist in this case.\\ 

Shilnikov proved that there are chaotic invariant sets in the neighborhood of
homoclinic orbits to a saddle-focus in three dimensional vector fields when the
magnitude of the real eigenvalue is larger than the magnitude of the real part
of the complex pair of eigenvalues~\cite{Shilnikov}. The homoclinic orbits of
the FitzHugh-Nagumo vector field satisfy this condition in the parameter region
$I$. Therefore, we expect to find many periodic orbits close to the homoclinic
orbits and parameters in $I$ with ``multi-pulse'' homoclinic orbits that have
several jumps connecting the left and right branches of the slow manifold
\cite{EvansFenichelFeroe}. Without making use of concepts from fast-slow
systems, Champneys et al. \cite{Sneydetal} described interactions of homoclinic
and periodic orbits that can serve to terminate curves of homoclinic
bifurcations. This provides an alternate perspective on identifying phenomena
that occur near the sharp turn of the C-curve in $I$. AUTO can be used to locate
families of periodic orbits that come close to a homoclinic orbit as their
periods grow. \\

\begin{figure}[htbp]
\centering
\includegraphics[width=0.85\textwidth]{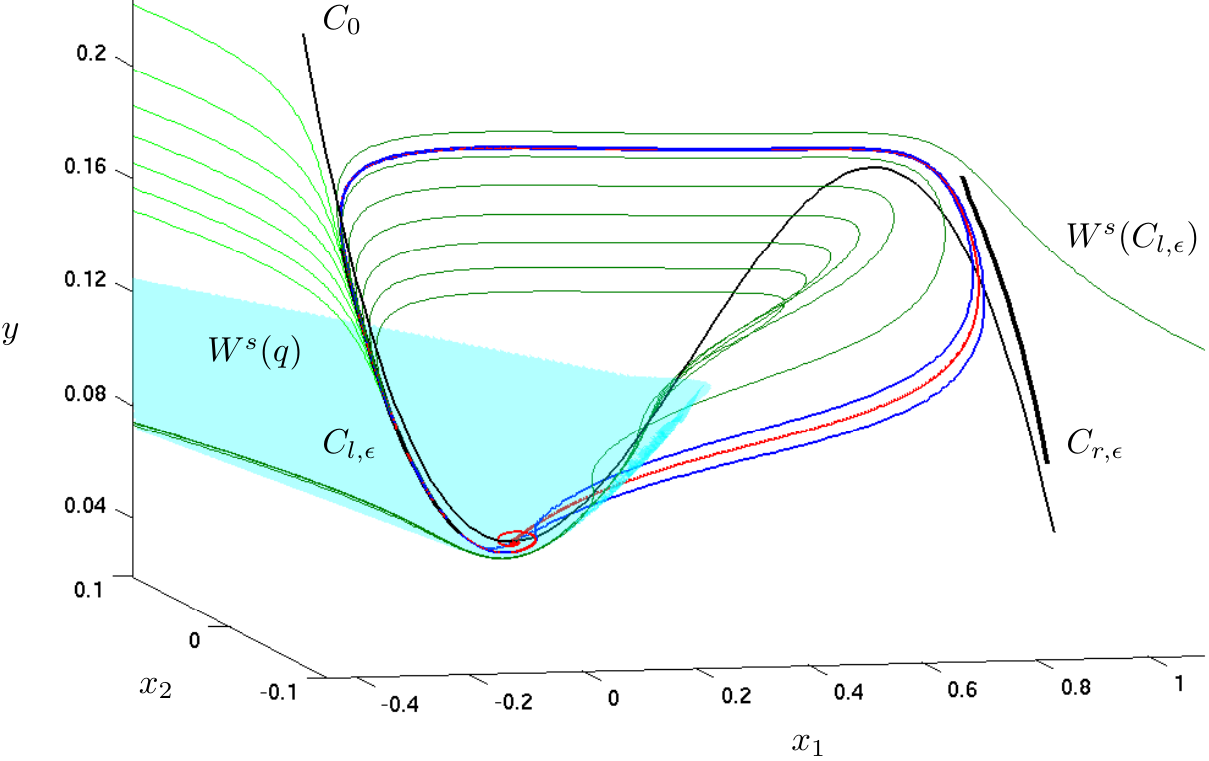}  
\caption{\label{fig:beforeend}Phase space along the C-curve near its sharp turn:
the parameter values $\epsilon=0.01$, $p=0.05$ and $s\approx 1.3254$ lie on the
C-curve. The homoclinic orbit (red), two periodic orbits born in the subcritical
Hopf (blue), $C_0$ (thin black), $C_{l,\epsilon}$ and $C_{r,\epsilon}$ (thick
black) are shown. The manifold $W^s(q)$ (cyan) has been truncated at a fixed
coordinate of $y$. Furthermore $W^s(C_{l,\epsilon})$ (green) is separated by
$C_{l,\epsilon}$ into two components shown here by dark green trajectories
interacting with $C_{m,\epsilon}$ and by light green trajectories that flow left
from $C_{l,\epsilon}$.}
\end{figure}

Figure \ref{fig:beforeend} shows several significant objects in phase space for
parameters lying on the C-curve. The homoclinic orbit and the two periodic
orbits were calculated using AUTO. The periodic orbits were continued in $p$
starting from a Hopf bifurcation for fixed $s\approx 1.3254$. Note that the
periodic orbit undergoes several fold bifurcations \cite{Sneydetal}. We show two
of the periodic orbits arising at $p=0.05$; see \cite{Sneydetal}. The
trajectories in $W^s(C_{l,\epsilon})$ have been calculated using a mesh on
$C_{l,\epsilon}$ and using backward integration at each mesh point and initial
conditions in the linear approximation $E^s(C_{l,\epsilon})$.\\

\begin{figure}[htbp]
\centering
\includegraphics[width=1\textwidth]{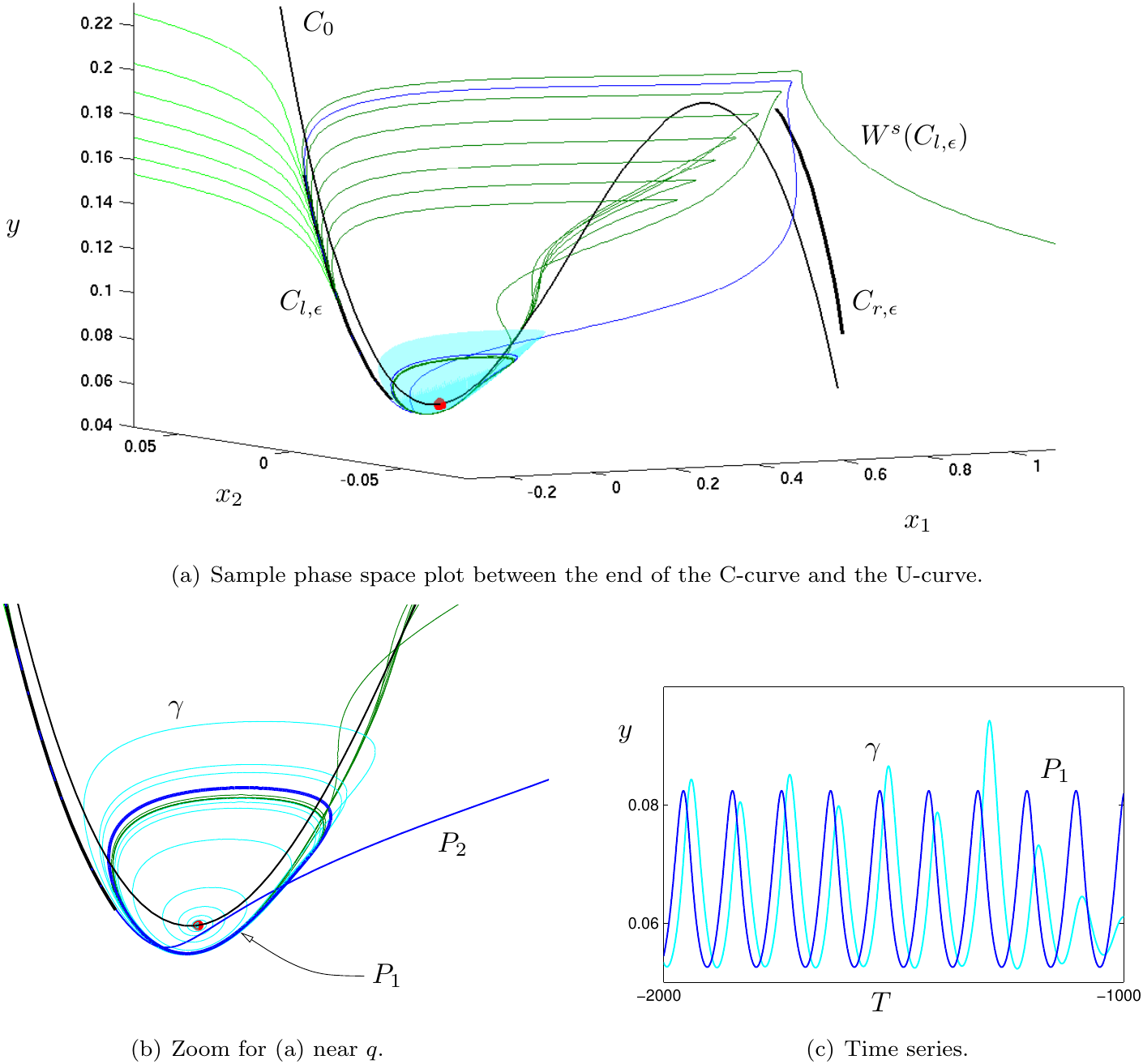}
\caption{The parameter values are $\epsilon=0.01$, $p=0.06$ and $s=1.38$. For
(a) we display two periodic orbits (blue), one with a single large excursion
$P_2$ and one consisting of a small loop $P_1$. We also show $q$ (red dot),
trajectories in $W^s(C_{l,\epsilon})$ (green) and $W^s(q)$ (cyan). In (b) a zoom
near $q$ is shown and we made plotted a single trajectory $\gamma\in W^s(q)$
(cyan). The plot (c) gives a time series of this trajectory $\gamma$ in
comparison to the periodic orbit $P_1$. Note that the trajectories are computed
backward in time, so the final points of the trajectories are on the left of the
figure. A phase shift of time along the periodic orbit would bring the two time
series closer.}
\label{fig:bigandsmall}
\end{figure} 

We observe from Figure \ref{fig:beforeend} that part of $W^s(q)$ lies near
$C_{l,\epsilon}$ as expected for (S2) to be satisfied. This is in contrast to
the situation beyond the turning of the C-curve shown in Figure
\ref{fig:bigandsmall} for $p=0.06$ and $s=1.38$. We observe that $W^s(q)$ is
bounded. Figure \ref{fig:bigandsmall}(a) shows two periodic orbits $P_1$ and $P_2$ obtained
from a Hopf bifurcation continuation starting for $s=1.38$ fixed. $P_2$ is of
large amplitude and is obtained after the first fold bifurcation occurred. $P_1$
is of small amplitude and is completely unstable. A zoom near $P_1$ in Figure
\ref{fig:bigandsmall}(b) and a time series comparison of a trajectory in $W^s(q)$ and
$P_1$ in Figure \ref{fig:bigandsmall}(c) show that
\be
\label{eq:alpha}
\lim_{\alpha}\{p:p\in W^s(q)\text{ and }p\neq q\}=P_1
\ee     
where $\lim_\alpha U$ denotes the $\alpha$-limit set of some set $U\subset
\R^m\times\R^n$. From \eqref{eq:alpha} we can also conclude that there is no
heteroclinic connection from $q$ to $P_1$ and only a connection from $P_1$ to
$q$ in a large part of the region $I$ beyond the turning of the C-curve. Since
$P_1$ is completely unstable, there can be no heteroclinic connections from $q$
to $P_1$. Therefore, double heteroclinic connections between a periodic orbit
and $q$ are restricted to periodic orbits that lie closer to the homoclinic
orbit than $P_1$. These can be expected to exist for parameter values near the
end of the C-curve in accord with  
the conjecture of Champneys et al. \cite{Sneydetal} and the ``Shilnikov''-model
presented in the next section.\\

\textit{Remark}: The recent manuscript
\cite{ChampneysKirkKnoblochOldemanRademacher} extends the results of
\cite{Sneydetal} that motivated this paper. A partial unfolding of a
heteroclinic cycle between a hyperbolic equilibrium point and a hyperbolic
periodic orbit is developed in \cite{ChampneysKirkKnoblochOldemanRademacher} .
Champneys et al. call this codimension two bifurcation an EP1t-cycle and the
point where it occurs in a two dimensional parameter space an EP1t-point. The
manuscript \cite{ChampneysKirkKnoblochOldemanRademacher} does not conclude
whether the EP1t-scenario occurs in the FitzHugh-Nagumo equation. The
relationship between the results of this paper and those of
\cite{ChampneysKirkKnoblochOldemanRademacher} have not yet been clarified.

\section{Homoclinic Bifurcations in Fast-Slow Systems}
\label{sec:contract}

It is evident from Figure \ref{fig:move} that the homoclinic orbits in the
FitzHugh-Nagumo equation exist in a very thin region in $(p,s)$-parameter space
along the C-curve. We develop a geometric model for homoclinic orbits that
resemble those in the FitzHugh-Nagumo equation containing segments of types
(S1), (F1), (S2) and (F2). The model will be seen to be an exponentially
distorted version of the Shilnikov model for a homoclinic orbit to a
saddle-focus~\cite{GH}. Throughout this section we assume that the parameters
lie in a region $I$ the region of the $(p,s)$-plane close to the upper turn of
the C-curve.\\

The return map of the Shilnikov model is constructed from two components: the
flow map past an equilibrium point, approximated by the flow map of a linear
vector field, composed with a regular map that gives a ``global return'' of the
unstable manifold of the equilibrium to its stable manifold~\cite{GH}. Place two
cross-sections $\Sigma_1$ and $\Sigma_2$ moderately close to the equilibrium
point and model the flow map from $\Sigma_1$ to $\Sigma_2$ via the linearization
of the vector field at the equilibrium.\\

\begin{figure}[htbp]
\centering
\includegraphics[width=0.70\textwidth]{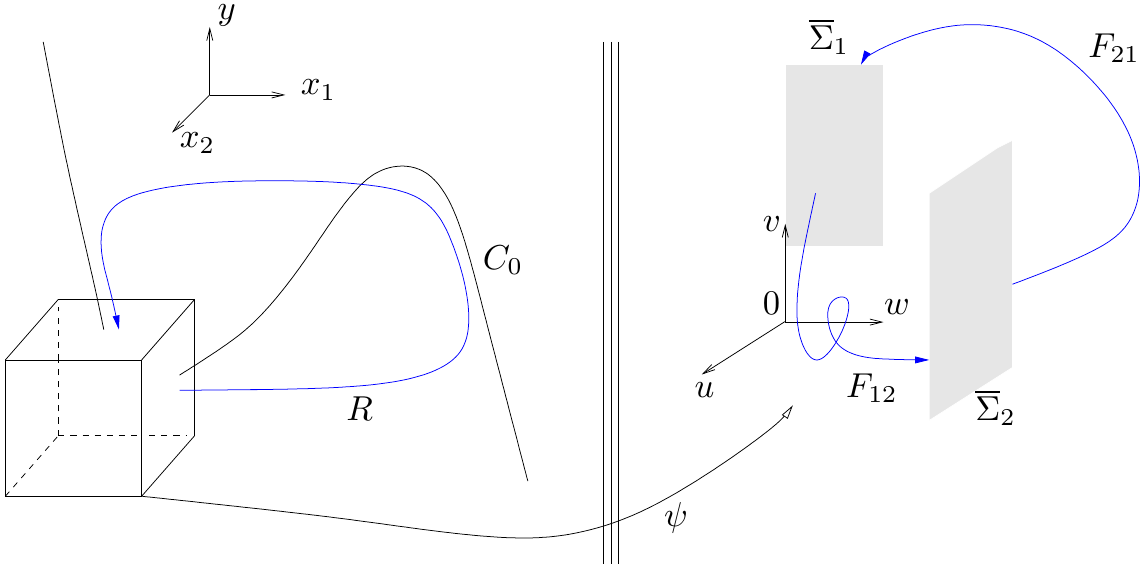}  
\caption{\label{fig:map1}Sketch of the geometric model for the homoclinic
bifurcations. Only parts of the sections $\overline{\Sigma}_i$ for $i=1,2$ are
shown.}
\end{figure}

The degree one coefficient of the characteristic polynomial at the equilibrium
has order $O(\epsilon)$, so the imaginary eigenvalues at the Hopf bifurcation
point have magnitude $O(\epsilon^{1/2})$. The real part of these eigenvalues
scales linearly with the distance from the Hopf curve. Furthermore we note that
the real eigenvalue of the equilibrium point remains bounded away from $0$ as
$\epsilon \to 0$.\\ 

Let $\psi(x_1,x_2,y)=(u,v,w)$ be a coordinate change near $q$ so that
$\psi(q)=0$ and the vector field is in Jordan normal form up to higher order
terms. We denote the sections obtained from the coordinate change into Jordan
form coordinates by $\overline{\Sigma}_1=\psi(\Sigma_1)$ and
$\overline{\Sigma}_2=\psi(\Sigma_2)$; see Figure \ref{fig:map1}. Then the vector
field is
\bea
\label{eq:linear}
u' & = & -\beta u - \alpha v \nonumber\\
v' & = & \alpha u - \beta v \qquad + \text{h.o.t.}\\
w' & = & \gamma\nonumber
\eea
with $\alpha,\beta,\gamma$ positive. We can choose $\psi$ so that the
cross-sections are $\overline{\Sigma}_1=\{u=0,w>0\}$ and
$\overline{\Sigma}_2=\{w=1\}$. The flow map $F_{12}:\overline{\Sigma}_1 \to
\overline{\Sigma}_2$ of the (linear) vector field \eqref{eq:linear} without
higher-order terms is given by
\be
 F_{12}(v,w) =
vw^{\beta/\gamma}\left(\cos\left(-\frac{\alpha}{\gamma}\ln(w)\right),
\sin\left(-\frac{\alpha}{\gamma}\ln(w)\right)\right)
\ee
Here $\beta$ and $\alpha$ tend to $0$ as $\epsilon \to 0$. The domain for
$F_{12}$ is restricted to the interval $v \in [\exp(-2 \pi \beta/\alpha),1]$
bounded by two successive intersections of a trajectory in $W^s(0)$ with the
cross-section $u=0$.\\

The global return map $R:\Sigma_2 \to \Sigma_1$ of the FitzHugh-Nagumo system is
obtained by following trajectories that have successive segments that are near
$W^s(C_{r,\epsilon})$ (fast), $C_{r,\epsilon}$ (slow), $W^u(C_{r,\epsilon}) \cap
W^s(C_{l,\epsilon})$ (fast), $C_{l,\epsilon}$ (slow) and $W^u(C_{l,\epsilon})$
(fast). The Exchange Lemma~\cite{JonesKopell} implies that the size of the
domain of $R$ in $\Sigma_2$ is a strip whose width is exponentially small. As
the parameter $p$ is varied, we found numerically that the image of $R$ has a
point of quadratic tangency with $W^s(q)$ at a particular value of $p$. We also
noted that $W^u(q)$ crosses $W^s(C_{r,\epsilon})$ as the parameter $s$ varies
\cite{GuckCK1}. Thus, we choose to model $R$ by the map
\be
(w,v) = F_{21}(u,v) = ( \sigma v + \lambda_2 - \rho^2 (u-\lambda_1)^2,
\rho(u-\lambda_1) + \lambda_3)
\ee
for $F_{21}$ where $\lambda_1$ represents the distance of $W^u(q) \cap \Sigma_2$
from the domain of $F_{21}$, $\lambda_2$ represents how far the image of
$F_{21}$ extends in the direction normal to $W^s(q)$, $\lambda_3$ is the $v$
coordinate of  $F_{21}(\lambda_1,0)$ and $\rho^{-1}$, $\sigma$ are
$O(e^{-K/\epsilon})$ for suitable $K>0$. We assume further that the domain of
$F_{21}$ is $[\lambda_1,\lambda_1 + \rho^{-1}] \times [-1,1]$. Figure
\ref{fig:map2} depicts $F_{21}$. With these choices, we observe two properties
of the C-curve of homoclinic orbits in the geometric model:
\begin{enumerate}
 \item 
If $\sigma v + \lambda_2 - \rho^2 (u-\lambda_1)^2$ is negative on the domain of
$F_{21}$, then the image of $F_{21}$ is disjoint from the domain of $F_{12}$ and
there are no recurrent orbits passing near the saddle point. Thus, recurrence
implies that $\lambda_2 > -\sigma$.
\item
If $\lambda_2 > 0$, then there are two values of $\lambda_1$ for which the
saddle-point has a single pulse homoclinic orbit. These points occur for values
of $\lambda_1$ for which the $w$-component of  $F_{21}(0,0)$ vanishes:
$\lambda_1 = \pm \rho^{-1} |\lambda_2|^{1/2}$. The magnitude of these values of
$\lambda_2$ is exponentially small.
\end{enumerate}
$\quad$\\

\begin{figure}[htbp]
\centering
\includegraphics[width=0.70\textwidth]{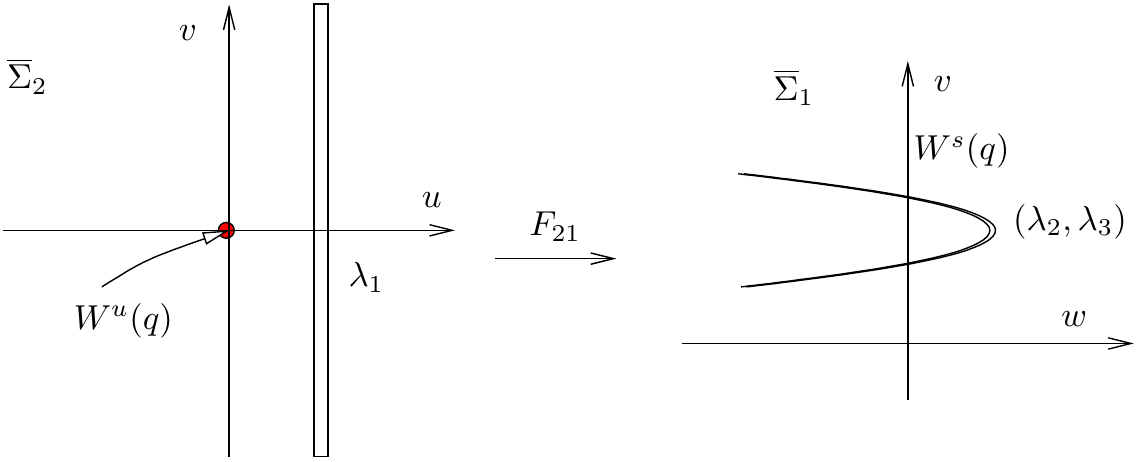}  
\caption{\label{fig:map2}Sketch of the map $F_{21}:\overline{\Sigma}_2\ra
\overline{\Sigma}_1$. The $(u,v)$ coordinates are centered at $W^u(q)$ and the
domain of $F_{21}$ is in the thin rectangle at distance $\lambda_1$ from the
origin. The image of this rectangle is the 
parabolic strip in $\overline{\Sigma}_1$.}
\end{figure}

When a vector field has a single pulse homoclinic orbit to a saddle-focus whose
real eigenvalue has larger magnitude than the real part of the complex
eigenvalues, Shilnikov \cite{Shilnikov} proved that a neighborhood of this
homoclinic orbit contains chaotic invariant sets. This conclusion applies to our
geometric model when it has a single pulse homoclinic orbit. Consequently, there
will be a plethora of bifurcations that occur in the parameter interval
$\lambda_2 \in [0,\sigma]$, creating the invariant sets as $\lambda_2$ decreases
from $\sigma$ to $0$.\\

The numerical results in the previous section suggest that in the
FitzHugh-Nagumo system, some of the periodic orbits in the invariant sets near
the homoclinic orbit can be continued to the Hopf bifurcation of the equilibrium
point. Note that saddle-node bifurcations that create periodic orbits in the
invariant sets of the geometric model lie exponentially close to the curve
$\lambda_2 = 0$ that models tangency of $W^s(q)$ and $W^u(C_{l,\epsilon})$ in
the FitzHugh-Nagumo model. This observation explains why the right most curve of
saddle-node bifurcations in Figure 7 of Champneys et al. \cite{Sneydetal} lies
close to the sharp turn of the C-curve.\\

There will also be curves of heteroclinic orbits between the equilibrium point
and periodic orbits close to the C-curve. At least some of these form
codimension two EP1t bifurcations near the turn of the C-curve as discussed by
Champneys et al. \cite{Sneydetal}. Thus, the tangency between $W^s(q)$ and
$W^u(C_{l,\epsilon})$ implies that there are several types of bifurcation curves
that pass exponentially close to the sharp turn of the C-curve in the
FitzHugh-Nagumo model. Numerically, any of these can be used to approximately
locate the sharp turn of the C-curve.

\section{Canards and Mixed Mode Oscillations}
\label{sec:further}

This section reports two additional observations about the FitzHugh-Nagumo model
resulting from our numerical investigations and analysis of the turning of the
C-curve. 

\subsection{Canard Explosion}
\label{sec:canardexplosion}

The previous sections draw attention to the intersections of $W^s(q)$ and
$W^u(C_{l,\epsilon})$ as a necessary component for the existence of homoclinic
orbits in the FitzHugh-Nagumo system. Canards for the backward flow of this
system occur along intersections of $W^u(C_{l,\epsilon})$ and $C_{m,\epsilon}$.
These intersections form where trajectories that track $C_{l,\epsilon}$ have
continuations that lie along $C_{m,\epsilon}$ which has two unstable fast
directions. We observed from Figures \ref{fig:beforeend} and
\ref{fig:bigandsmall} that a completely unstable periodic orbit born in the Hopf
bifurcation on the U-curve undergoes a canard explosion, increasing its
amplitude to the size of a relaxation oscillation orbit upon decreasing $p$.
This canard explosion happens very close to the intersections of
$W^u(C_{l,\epsilon})$ and $C_{m,\epsilon}$.\\

\begin{figure}[htbp]
\centering
\includegraphics[width=0.9\textwidth]{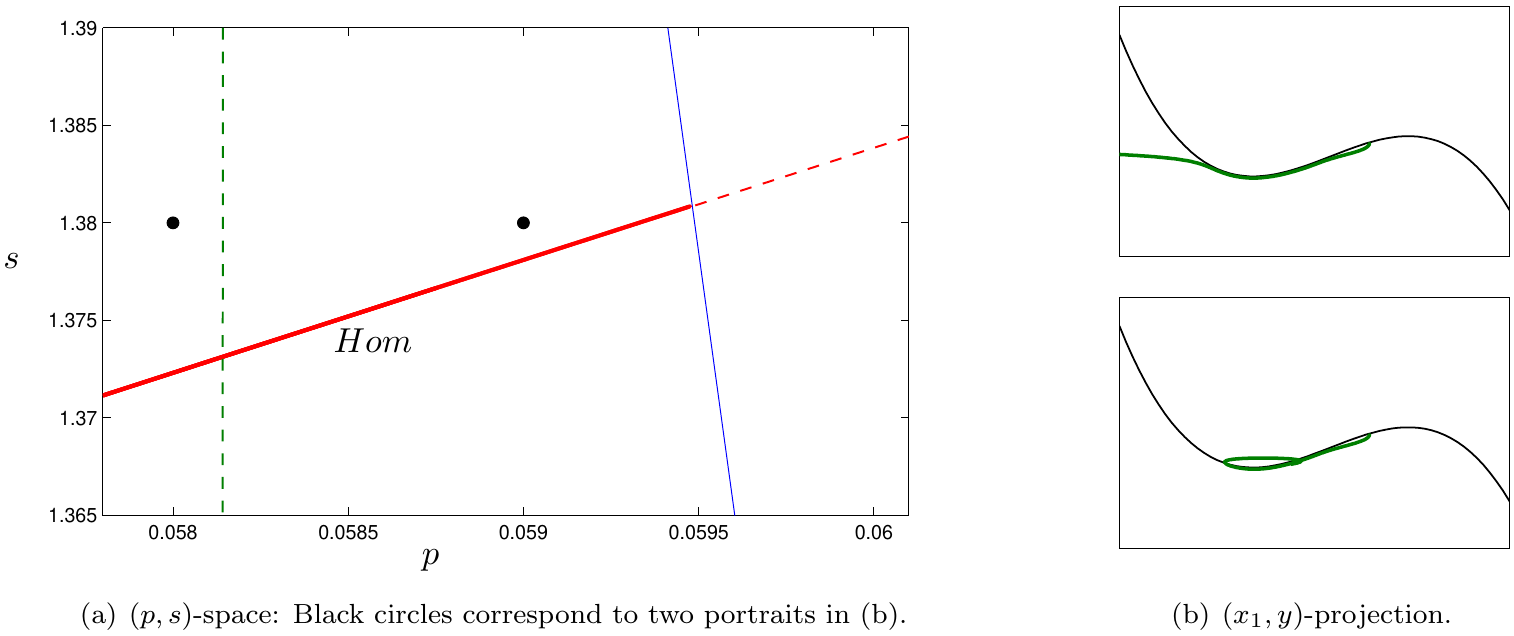}
\caption{\label{fig:canards}The dashed green curve indicates where canard orbits
start to occur along $C_{m,\epsilon}$. For values of $p$ to the left of the
dashed green curve we observe that orbits near the middle branch escape in
backward time (upper panel in (b)). For values of $p$ to the right of the dotted
green curve trajectories near $C_{m,\epsilon}$ stay bounded in backward time.}
\end{figure}

To understand where this transition starts and ends we computed the middle
branch $C_{m,\epsilon}$ of the slow manifold by integrating backwards from
points between the fold points $x_{1,-}$ and $x_{1,+}$ starting close to
$C_{m,0}$ and determined which side of $W^u(C_{l,\epsilon})$ these trajectories
came from. The results are shown in Figure \ref{fig:canards}. The dashed green
curve divides the $(p,s)$ plane into regions where the trajectory that flows
into $C_{m,\epsilon}$ lies to the left of $W^u(C_{l,\epsilon})$ and is unbounded
from the region where the trajectory that flows into $C_{m,\epsilon}$ lies to
the right of $W^u(C_{l,\epsilon})$ and comes from the periodic orbit or another
bounded invariant set. This boundary was found by computing trajectories
starting on $C_{m,0}$ backward in time. In backward time the middle branch of
the slow manifold is attracting, so the trajectory first approaches
$C_{m,\epsilon}$ and then continues beyond its end when $x_1$ decreases below
$x_{1,-}$. Figure \ref{fig:canards} illustrates the difference in the behavior of
these trajectories on the two sides of the dashed green curve. It shows that the 
parameters with canard orbits for the backward
flow have smaller values of $p$ than those for which $W^s(q)$ and
$W^u(C_{l,\epsilon})$ have a tangential intersection. The turns of the C-curve
do not occur at parameters where the backward flow has canards.\\ 

\subsection{Mixed-Mode Oscillations}
\label{sec:mmo}

Mixed-mode oscillations (MMOs) have been observed in many fast-slow systems; see
e.g.
\cite{MilikSzmolyan,RotsteinWechselbergerKopell,RubinWechselberger,
GuckenheimerSH}. MMOs are periodic orbits which consist of sequences of small
and large amplitude oscillations. The notation $L^s$ is used to indicate an MMO
with $L$ large and $s$ small oscillations.\\

\begin{figure}[htbp]
\centering
\includegraphics[width=0.9\textwidth]{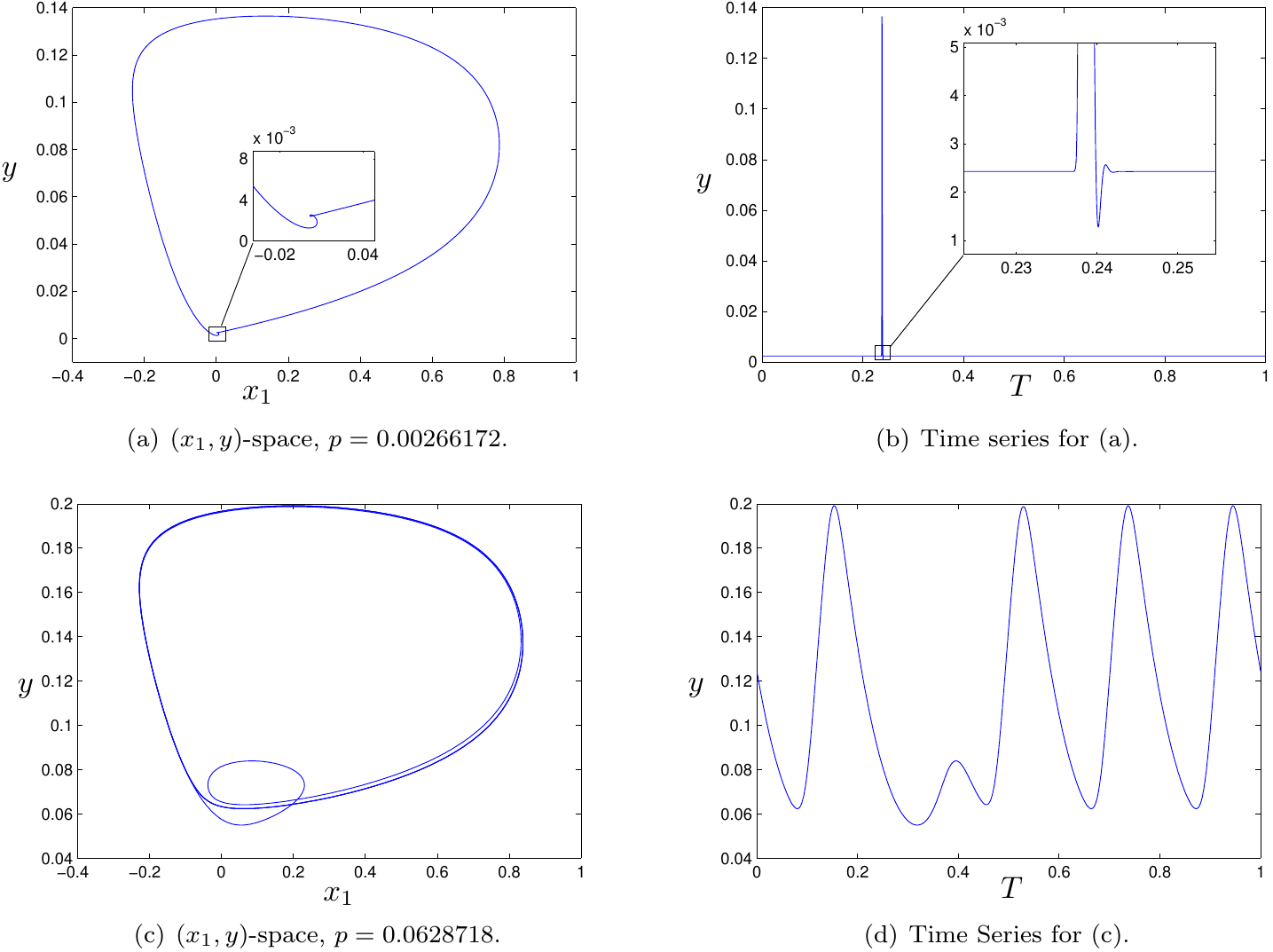}
\caption{\label{fig:mmo}Some examples of mixed-mode oscillations in the
FitzHugh-Nagumo equation. Fixed parameter values are $\epsilon=0.01$ and $s=1$.
Note that the period of the orbits has been rescaled to $1$ in $(b)$ and $(d)$.}
\end{figure}

The FitzHugh-Nagumo equation \eqref{eq:fhn} exhibits MMOs: the periodic orbits
close to the homoclinic orbit make small oscillations near the equilibrium point
in addition to large amplitude relaxation oscillations. A $1^1$ MMO is shown in
Figure \ref{fig:mmo}(a)-(b). It was obtained by switching from the homoclinic
C-curve to a nearby curve of periodic orbits in a continuation framework. Note
that the existence of multi-pulse homoclinic orbits near a Shilnikov homoclinic
orbit \cite{GonchenkoTuraevGaspardNicolis,EvansFenichelFeroe} implies that much
more complicated patterns of MMOs also exist near the homoclinic C-curve. $L^s$
MMOs with very large $L$ and $s$ near the homoclinic C-curve are theoretically
possible although observing them will be very difficult due to the exponential
contraction described in Section \ref{sec:contract}.\\  

In addition to the MMOs induced by the Shilnikov bifurcation we also find MMOs
which exist due to orbits containing canard segments near the completely
unstable slow manifold $C_{m,\epsilon}$. An example of a $4^1$ MMO is shown in
Figure \ref{fig:mmo}(c)-(d) obtained by continuation. In this case the small
oscillations arise due to small excursions reminiscent to MMOs in three-time
scale systems \cite{JalicsKrupaRotstein,KrupaPopovicKopellRotstein}. MMOs of
type $L^1$ with $L=1,2,3,\ldots,O(10^2)$ can easily be observed from
continuation and we expect that $L^1$ MMOs exist for any $L\in \N$. It is likely
that these MMOs can be analyzed using a version of the FitzHugh-Nagumo equation
containing $O(1)$, $O(\epsilon)$ and $O(\epsilon^2)$ terms similar to the one
introduced in \cite{GuckCK1} but we leave this analysis for future work.\\

Figure \ref{fig:mmo} was obtained by varying $p$ for fixed values of
$\epsilon=0.01$ and $s=1$. Thus, varying a single parameter suffices to switch
between MMOs whose small amplitude
oscillations have a different character. In the first case, the small amplitude
oscillations occur when the orbit comes close to a saddle focus rotating around
its stable manifold, while in the second case, the trajectory never approaches
the equilibrium and its small amplitude oscillations
occur when the trajectory flows along the completely unstable slow manifold
$C_{m,\epsilon}$. Different types of MMOs seem to occur very frequently in
single- and multi-parameter bifurcation problems; see
\cite{DesrochesKrauskopfOsinga2} for a recent example. This contrasts with most
work on the analysis of MMOs \cite{Koper-vdP,RotsteinWechselbergerKopell} that
focuses on identifying \emph{the} mechanism for generating MMOs in an example.
The MMOs in the FitzHugh-Nagumo equation show that a fast-slow system with three
or more variables can exhibit MMOs of different types and that one should not
expect a priori that a single mechanism suffices to explain all the MMO
dynamics.

\bibliographystyle{plain}

\end{document}